\documentclass[a4paper,10pt]{amsart}

\usepackage{fullpage}

\newtheorem{theorem}{Theorem}
 
\theoremstyle{definition}

\newtheorem*{remark}{Remark}

\newcommand{\N}{\mathbb{N}}

\newcommand{\C}{{\mathbb{C}}}

\newcommand{\dd}{{\rm d}}

\begin{document}

\title{A sharp form of the discrete Hardy
inequality and the Keller--Pinchover--Pogorzelski inequality}

\author{David Krej{\v c}i{\v r}{\' i}k}
\author{Franti{\v s}ek {\v S}tampach}
\address{Department of Mathematics, Faculty of Nuclear Sciences and Physical Engineering, Czech Technical University 
in Prague, Trojanova 13, 12000 Prague~2, Czech Republic}
\email{david.krejcirik@fjfi.cvut.cz \& stampfra@fjfi.cvut.cz}

\begin{abstract}
We give a short proof of a recently established Hardy-type inequality
due to Keller, Pinchover, and Pogorzelski together with its optimality.
Moreover, we identify the remainder term which makes it into an identity.
\end{abstract}
\maketitle

\noindent 
In~\cite{keller-etal_amm18}, Keller, Pinchover, and Pogorzelski made an interesting observation that the celebrated discrete Hardy inequality
\[
 \sum_{n=1}^{\infty}\frac{|u_{n}|^{2}}{4n^{2}}\leq\sum_{n=1}^{\infty}|u_{n}-u_{n-1}|^{2},
\]
which holds true for all complex sequences $\{u_{n}\}_{n=0}^{\infty}$ with $u_{0}=0$, can be improved to the inequality
\begin{equation}
 \sum_{n=1}^{\infty}w_{n}|u_{n}|^{2}\leq\sum_{n=1}^{\infty}|u_{n}-u_{n-1}|^{2}
\label{eq:hardy_improved}
\end{equation}
with the weight sequence
\begin{equation}
 w_{n}:=2-\sqrt{\frac{n+1}{n}}-\sqrt{\frac{n-1}{n}}>\frac{1}{4n^{2}}.
\label{eq:optimal_weight}
\end{equation}
Moreover, in~\cite{keller-etal_cmp18}, the same authors showed that the improved discrete Hardy inequality is actually optimal. This means that if~\eqref{eq:hardy_improved} holds with a weight sequence $\tilde{w}$ such that $\tilde{w}_{n}\geq w_{n}$ for all $n\in\N$, then necessarily $\tilde{w}_{n}=w_{n}$ for all $n\in\N$. Note that if inequality~\eqref{eq:hardy_improved} holds true for all finitely supported sequences $u\in C_{0}(\N_{0})$, then it is true for all complex sequences~$u$. 

The aim of this note is to provide a short and elementary proof of inequality~\eqref{eq:hardy_improved} as well as its optimality. In fact, we show more, namely, we determine a remainder term in~\eqref{eq:hardy_improved} yielding an identity 
from which the inequality~\eqref{eq:hardy_improved} readily follows.

\begin{theorem}
 For all $u\in C_{0}(\N_{0})$ with $u_{0}=0$, we have the identity
 \begin{equation}
 \sum_{n=1}^{\infty} w_{n} |u_n|^2  + \sum_{n=2}^\infty \left|\sqrt[4]{\frac{n-1}{n}} \, u_n - \sqrt[4]{\frac{n}{n-1}} \, u_{n-1} \right|^2 =
 \sum_{n=1}^{\infty} |u_n - u_{n-1}|^2.
 \label{eq:ident}
 \end{equation}
  In particular, the inequality~\eqref{eq:hardy_improved} holds and, in addition, is optimal.
\end{theorem}

\begin{proof}
First, we establish~\eqref{eq:ident}. An important observation for the proof is that the weight~\eqref{eq:optimal_weight} can be expressed as
\[
  w_n = \frac{h_n - h_{n+1}}{\sqrt{n}}
\]
for $h_n := \sqrt{n} - \sqrt{n-1}$. 
Since
\begin{align*}
  &\left|
  \sqrt{1-\frac{h_n}{\sqrt{n}}} \, u_n 
  - \sqrt{1+\frac{h_n}{\sqrt{n-1}}} \, u_{n-1}
  \right|^2 
  \\
  &\hskip68pt= \left(1-\frac{h_n}{\sqrt{n}}\right) \, |u_n|^2
  + \left(1+\frac{h_n}{\sqrt{n-1}}\right) \, |u_{n-1}|^2
  -2 \, \Re (\bar{u}_n u_{n-1})
  \\
  &\hskip68pt=
  |u_n - u_{n-1}|^2 
  - h_n \left(
  \frac{|u_n|^2}{\sqrt{n}} - \frac{|u_{n-1}|^2}{\sqrt{n-1}}
  \right)
\end{align*}
we have the identity
\begin{equation}
\left|
  \sqrt{1-\frac{h_n}{\sqrt{n}}} \, u_n 
  - \sqrt{1+\frac{h_n}{\sqrt{n-1}}} \, u_{n-1}
  \right|^2 \!\!+ 
  h_n \!\left(
  \frac{|u_n|^2}{\sqrt{n}} - \frac{|u_{n-1}|^2}{\sqrt{n-1}}
  \right)\!=
  |u_n - u_{n-1}|^2
\label{eq:key_id}
\end{equation}
for all $n\in\N$ and $\{u_{n}\}_{n=0}^{\infty}\subset\C$ with $u_{0}=0$, where the terms 
\[
 \sqrt{1+\frac{h_n}{\sqrt{n-1}}}\,u_{n-1} \quad \mbox{ and } \quad \frac{|u_{n-1}|^2}{\sqrt{n-1}}
\]
are to be understood as zeros if $n=1$. Next, summing by parts, we obtain
\[
 \sum_{n=1}^{\infty}w_{n}|u_{n}|^{2}=\sum_{n=1}^{\infty}\frac{h_{n}-h_{n+1}}{\sqrt{n}}|u_{n}|^{2}=\sum_{n=1}^{\infty}h_n\left(
  \frac{|u_n|^2}{\sqrt{n}} - \frac{|u_{n-1}|^2}{\sqrt{n-1}}
  \right)
\] 
for any $u\in\C_{0}(\N_{0})$ with $u_{0}=0$. Thus, applying~\eqref{eq:key_id}, we establish~\eqref{eq:ident}.

Second, we prove the optimality. Suppose $\tilde{w}$ is a sequence such that $\tilde{w}_{n}\geq w_{n}$ for all $n\in\N$ and~\eqref{eq:hardy_improved} holds for all $u\in C_{0}(\N_{0})$, $u_{0}=0$, with $w$ replaced by $\tilde{w}$. Then, by~\eqref{eq:ident},
\begin{equation}
 0\leq\sum_{n=1}^{\infty}(\tilde{w}_{n}-w_{n})|u_{n}|^{2}\leq \sum_{n=2}^\infty \left|\sqrt[4]{\frac{n-1}{n}} \, u_n - \sqrt[4]{\frac{n}{n-1}} \, u_{n-1} \right|^2
\label{eq:towards_optimality}
\end{equation}
for all $u\in C_{0}(\N_{0})$. Note the right-hand side of~\eqref{eq:towards_optimality} vanishes if $u_{n}=\sqrt{n}$ which, however, is not a finitely supported sequence. Therefore we regularize the sequence $u_n = \sqrt{n}$ by introducing $u_n^N := \xi_n^N \sqrt{n}$, where
\[
  \xi_n^N :=
  \begin{cases}
    1 & \mbox{if} \quad n < N \,,
    \\[2pt]
    \frac{\displaystyle 2 \log{N}-\log n}{\displaystyle\log N}
    & \mbox{if} \quad N\leq n \leq N^2 \,,
    \\[2pt]
    0 & \mbox{if} \quad n > N^2 \,,
  \end{cases}
\]
with $N\geq 2$. 
Notice that $\xi^N\to1$ pointwise as $N \to \infty$.
At the same time, we have
\begin{align*}
&\sum_{n=2}^\infty \left|\sqrt[4]{\frac{n-1}{n}} \, u_{n}^{N} - \sqrt[4]{\frac{n}{n-1}} \, u_{n-1}^{N} \right|^2
=\sum_{n=2}^\infty\sqrt{n(n-1)}\left|\xi_{n}^{N}-\xi_{n-1}^{N}\right|^{2}\\
&=\frac{1}{\log^{2}N}\sum_{n=N+1}^{N^{2}}\sqrt{n(n-1)}\log^{2}\left(\frac{n}{n-1}\right)
\leq\frac{1}{\log^{2}N}\sum_{n=N+1}^{N^{2}}\frac{\sqrt{n(n-1)}}{(n-1)^{2}}\\
&\leq\frac{2}{\log^{2}N}\sum_{n=N+1}^{N^{2}}\frac{1}{n-1}\leq\frac{2}{\log^{2}N}\int_{N}^{N^{2}}\frac{\dd n}{n-1}=\frac{2\log(N+1)}{\log^{2}N}\leq\frac{4}{\log N}.
\end{align*}
Since the last expression tends to $0$ for $N\to\infty$, we deduce from~\eqref{eq:towards_optimality} that 
\[
\sum_{n=1}^{\infty}n(\tilde{w}_{n}-w_{n})=0.
\]
Bearing in mind that $\tilde{w}_{n}\geq w_{n}$ for all $n\in\N$, we conclude that $w_{n}=\tilde{w}_{n}$ for all $n\in\N$.
\end{proof}

\begin{remark}
Inequality~\eqref{eq:hardy_improved} is a particular case of the improved discrete $p$-Hardy inequality established by Fischer, Keller, and Pogorzelski in a currently unpublished preprint~\cite{fischer-etal_preprint}. Their weight sequence again improves upon the one of the classical discrete $p$-Hardy inequality but its optimality has not been demonstrated yet. A short proof of the classical discrete $p$-Hardy inequality with an optimal constant can be found in the recent paper~\cite{lefevre_am20}.
\end{remark}

\section*{Acknowledgment.}
The research of the authors was supported by the GA{\v C}R grant No.~20-17749X.

\end{document}